\documentclass[a4paper, 11pt]{article}

\usepackage
{
makeidx,epsfig, amstext, setspace, amsmath, amssymb,amsthm,amsfonts,times,latexsym,mathptmx,algorithm,algorithmic,
wrapfig,bbm,graphicx,color,enumerate,endnotes
}

\usepackage[top=2.5cm, bottom= 2.5cm, left= 2.5cm, right= 2.5cm]{geometry}
\newtheorem{theorem}{Theorem}

\newtheorem{LE}[theorem]{Lemma}

\newtheorem{OB}[theorem]{\bf Observation}
\newtheorem{Claim}[theorem]{Claim}

\newtheorem*{DEF}{Definition}

\newcounter{claim_nb}[theorem]
\newtheorem*{claim*}{Claim}

\newcommand{\zT}{\mathcal T}
\newcommand{\zX}{\mathcal X}

\newcommand{\ignore}[1]{}
 
\newenvironment{cproof}
{\begin{proof}
 [Proof.]
 \vspace{-1.5\parsep}
}
{ \end{proof}}

\title{Immersions in highly edge connected graphs}

\author{
D\'aniel Marx\thanks{Institute for Computer Science and Control, Hungarian Academy of Sciences (MTA SZTAKI), Budapest, Hungary
	\texttt{dmarx@cs.bme.hu}.
Research partially supported by the European Research Council (ERC)  grant 
``PARAMTIGHT: Parameterized complexity and the search for tight
complexity results,'' reference 280152 and OTKA grant NK105645.
}
 \and
Paul Wollan\thanks{Department of Computer Science, University of Rome, ``La Sapienza", Rome, Italy \texttt{wollan@di.uniroma1.it}.  Partially supported by the European Research Council under the European Unions Seventh Framework Programme (FP7/2007-2013)/ERC Grant Agreement no. 279558}
}

\begin{document}
\maketitle
%
\begin{abstract}
We consider the problem of how much edge connectivity is necessary to force a graph $G$ to contain a fixed graph $H$ as an immersion.  We show that if the maximum degree in $H$ is $\Delta$, then all the examples of $\Delta$-edge connected graphs which do not contain $H$ as a weak immersion must have a tree-like decomposition called a tree-cut decomposition of bounded width.  If we consider strong immersions, then it is easy to see that there are arbitrarily highly edge connected graphs which do not contain a fixed clique $K_t$ as a strong immersion.  We give a structure theorem which roughly characterizes those highly edge connected graphs which do not contain $K_t$ as a strong immersion.
\end{abstract}
%
\section{Introduction}

We consider graphs with parallel edges but no loops.  In this article, we will examine the immersion relation between graphs.
\begin{DEF}
A graph $G$ admits an \emph{immersion} of a graph $H$ if there exists a function $\pi$ with domain $V(H) \cup E(H)$ mapping to the set of connected subgraphs of $G$ which satisfies the following: 
\begin{itemize}
\item[a.] for all $v \in V(H)$, $\pi(v)$ is a vertex of $G$, and if $u \in V(H)$, $u \neq v$, then $\pi(u) \neq \pi(v)$;
\item[b.] for every edge $f \in E(H)$ with endpoints $x$ and $y$, $\pi(f)$ is a path with endpoints equal to $\pi(x)$ and $\pi(y)$;
\item[c.] for edges $f, f' \in E(H)$, $f \neq f'$, $\pi(f)$ and $\pi(f')$ have no edge in common.
\end{itemize}
The vertices $\{\pi(x): x \in V(H)\}$ are the \emph{branch vertices} of the immersion.  We will also say that $G$ immerses $H$ or alternatively that $G$ contains $H$ as an immersion.  The edge-disjoint paths $\pi(f)$ for $f\in E(H)$ are the \emph{composite paths} of the immersion.
\end{DEF}
One can distinguish between \emph{strong} and \emph{weak} immersions.  The definition given above is that of a weak immersion; in a strong immersion, one additionally requires that no branch vertex be contained as an internal vertex of a composite path.  We will consider both forms of immersions in this article.  In the interest of brevity, we will typically refer to weak immersions as simply ``immersions", and explicitly specify strong immersions whenever we are focusing on strong immersions.

There is an easy structure theorem for graphs excluding a fixed $H$ as an immersion \cite{W1}, \cite{DMMS}.  If we fix the graph $H$ and let $\Delta$ be the maximum degree of a vertex in $H$, then one obvious obstruction to a graph $G$ containing $H$ as an immersion is if every vertex of $G$ has degree less than $\Delta$.  The structure theorem shows that this is approximately the only obstruction.  The structure theorem says that for all $t \ge 1$, any graph which does not have an immersion of $K_t$ can be decomposed into a tree-like structure of pieces with at most $t$ vertices each of degree at least $t^2$.  

The value $t^2$ cannot be significantly improved.  Consider the graph $P_{k, n}$ to be the graph obtained from taking a path on $n$ vertices and adding $k-1$ additional parallel edges to each edge.  See Figure \ref{fig1}.  
\begin{figure}[htb]\label{fig1}
\begin{center}
\includegraphics[scale = .5]{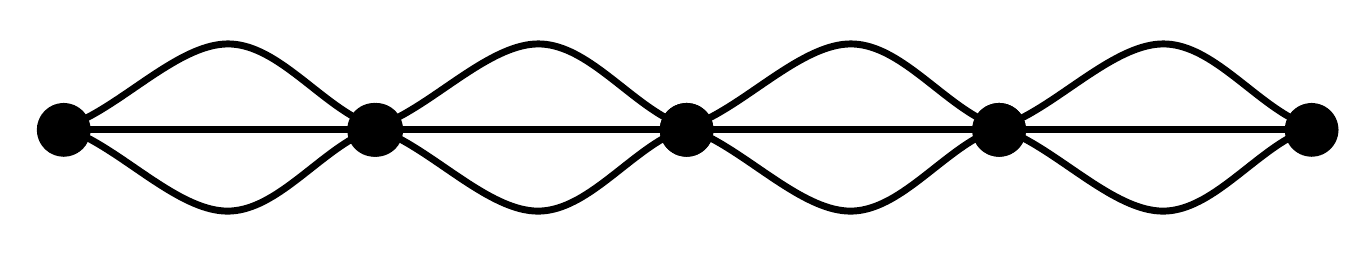}
 \caption{The graph $P_{3,4}$.}
 \end{center}
\end{figure}
The graphs $P_{k, n}$ are natural examples of graphs which are highly edge connected and exclude a given immersion or strong immersion.  Observe that the graph $P_{t^2/4-1, n}$ is roughly $t^2/4$-edge connected, but does not contain $K_t$ as an immersion.  

The structure theorem for weak immersions gives rise to a variant of tree decompositions based on edge cuts instead of vertex cuts, called \emph{tree-cut decompositions}.  The minimum width of a tree-cut decomposition is the tree-cut width of a graph.  We give the exact definition of these notions in the following section.   The example above of a highly edge connected graph with no $K_t$ immersion has tree-cut width bounded by a function of $t$.  Thus, one might hope that all the highly edge connected graphs which do not admit $K_t$ as an immersion similarly have bounded tree-cut width.  This is the main result of this article. 

\begin{theorem}\label{thm:main}
There exists a function $g$ satisfying the following.  Let $k \ge 4, n \ge 1$ be positive integers.  Then for all graphs $H$ with maximum degree $k$ on $n$ vertices and for all $k$-edge connected graphs $G$, either $G$ admits an immersion of $H$, or $G$ has tree-cut width at most $g(k,n)$.
\end{theorem}

The $k = 4$ case of Theorem \ref{thm:main} was proven by Chudnovsky, Dvorak, Klimosova, and Seymour \cite{CDKS};  our proof builds on this result to show the statement for general $k\ge 4$.

The theorem is not true for $k = 3$.  This is because if $G$ and $H$ are 3-regular graphs, then $G$ contains $H$ as an immersion if and only if $G$ contains $H$ as a topological minor.  Thus, if $H$ is any $3$-regular graph which cannot be embedded in the plane, then any 3-regular planar graph $G$ cannot contain $H$ as an immersion and such graphs can have arbitrarily large tree-width. 

Recent work has proven several special cases of Theorem \ref{thm:main} while generalizing the statement to strong immersions.  Giannopolous, Kaminski, and Thilikos \cite{GKT} have shown that for $k \ge 4$, every $k$-edge connected graph embedded in a surface of bounded genus either contains a fixed $H$ of maximum degree $k$ as a strong immersion, or has tree-width bounded by a function of $H$ and the genus of the surface.  Dvorak and Klimosova showed \cite{DK} that the $k=4$ case of Theorem \ref{thm:main} holds for strong immersions as well when the graph $G$ is assumed to be simple.  It is unclear whether Theorem \ref{thm:main} might also be true in general for strong immersions in simple graphs. 

Theorem \ref{thm:main} can also be contrasted with recent work of Norine and Thomas on clique minors in large $k$-connected graphs.  They have announced a proof that for every $k \ge 5$, every sufficiently large $k$-connected graph $G$ either contains $K_k$ as a minor or there exists a set $X\subseteq V(G)$ with $|X| = k-5$ such that $G-X$ is planar.  Here we see similarly that every $k$-edge connected graph which does not admit an immersion of $K_{k+1}$ falls into a relatively simple class of graphs (although the proofs here are dramatically easier than the proof of Norine and Thomas' result).  Note that Theorem \ref{thm:main} cannot be improved further to simply bound the size of the graph $G$.  Again consider that the graph $P_{k,n}$ is $k$-edge connected and can be chosen to have arbitrarily many vertices, but will not admit an immersion of $K_{k+1}$.  

If we instead consider the strong immersion relation, then it has recently been shown that in simple graphs with no parallel edges, any graph with minimum degree 200t contains a strong immersion of $K_t$ by DeVos, Dvorak, Fox, McDonald, Mohar, and Scheide \cite{DDFMMS}.  This bound is clearly best possible, up to improvements in the constant 200.  However, if we consider graphs which do possibly have parallel edges, then there exist arbitrarily highly edge connected graphs which do not contain a fixed strong immersion.  For example, for all fixed $t$, the graphs $P_{k, n}$ are $k$-edge connected and do not contain a strong immersion of $K_3$ for all $k$ and $n$ since paths linking a pair of vertices in $P_{k, n}$ must pass through all the vertices in between.  In the second main result of this article, we will see that such long paths of parallel edges are essentially the only obstructions to highly edge connected graphs containing a fixed clique as a strong immersion.  We leave the exact statement until Section \ref{sec:strong}.  The proof proceeds by starting from the edge bound in \cite{DDFMMS} and analyzing how parallel edges can be contained in the graph.

We quickly outline how the article will proceed.  In Section \ref{sec:tcd}, we give the definitions of tree-cut decompositions and the tree-cut width of a graph as well as state the structure theorem for immersions.  In Section \ref{sec:tangles}, we introduce several important graph minors tools which we will use going forward, including a necessary result of Robertson and Seymour on tangles, and look at a packing result for subgraphs called spiders.  In Section \ref{sec:spider}, we give an exact characterization when a given graph and tangle contain a spider.  In Section \ref{sec:bdeddeg}, we see that any $k$-edge connected graph which does not contain an immersion of a given $H$ of max degree $k$ must essentially have bounded degree.  In Section \ref{sec:thm1}, we introduce another tool of Robertson and Seymour on finding disjoint paths given the presence of a clique minor, and proceed to give the proof of Theorem \ref{thm:main}.  Finally, in Section \ref{sec:strong}, we turn our attention to strong immersions and state and prove the structure theorem for highly edge connected graphs which do not have a strong immersion of $K_t$ for a fixed value $t$. 

We conclude with some notation.  Let $G$ be a graph and $v \in V(G)$.  The degree $\deg(v)$ is the number of edges incident with $v$ and $\Delta(G)$ is the maximum degree of a vertex in $G$.   The \emph{neighborhood of $v$} is the set of vertices adjacent to $v$ and is denoted $N(v)$.  Note that $\deg(v) \ge |N(v)|$, however given the possibility of parallel edges, it is not necessarily true that equality holds.  A graph is \emph{simple} if it has no parallel edges.  Let $X \subseteq V(G)$.  The set of edges with exactly one endpoint in $X$ is denoted $\delta(X)$.  We will use $\delta(v)$ for $\delta(\{v\})$.  The set of vertices of $V(G) \setminus X$ with a neighbor in $X$ is denoted $N_G(X)$, or simply $N(X)$ when there can be no confusion.  We refer to the graph induced on $X$ by $G[X]$.   We use $G-X$ to refer to the graph induced on $V(G) \setminus X$.  For a subset $F \subseteq E(G)$ of edges, we use $G-F$ to refer to the graph $(V(G), E(G) \setminus F)$.  For subgraphs $G_1$ and $G_2$ of $G$, the subgraph $G_1 \cup G_2$ has vertex set $V(G_1) \cup V(G_2)$ and edge set $E(G_1) \cup E(G_2)$.  We will use $G-x$ as shorthand notation for $G-\{x\}$ when $x$ is a single element of either $V(G)$ or $E(G)$.  Finally, we will often want to reduce $G$ to a smaller graph by identifying a subset of vertices to a single vertex.  Let $X \subseteq V(G)$, define $G'$ be the graph obtained by deleting every edge with both endpoints in $X$ and identifying the vertex set $X$ to a single vertex.  We will say that $G'$ is obtained from $G$ by \emph{consolidating} $X$.  {\em Contraction} is the special case of this operation when $X$ induces a connected graph. Note that consolidating (contracting) a vertex set $X$ can create parallel edges if there is a vertex $v\not \in X$ with more than one edge into $X$. 

\section{Tree-cut decompositions}\label{sec:tcd}

In this section, we give the necessary definitions of tree-cut decompositions as well as state the structure theorem for graphs excluding a fixed clique immersion.  A \emph{near-partition} of a set $X$ is a family of subsets $X_1, \dots, X_k$, possibly empty, such that $\bigcup_1^k X_i = X$ and $X_i \cap X_j = \emptyset$ for all $1 \le i < j \le k$.  

\begin{DEF}
A \emph{tree-cut decomposition} of a graph $G$ is a pair $(T, \zX)$ such that $T$ is a tree and $\zX = \{X_t \subseteq V(G): t \in V(T)\}$ is a near-partition of the vertices of $G$ indexed by the vertices of $T$.  For each edge $e = uv$ in $T$, $T-uv$ has exactly two components, namely $T_v$ and $T_u$ containing $v$ and $u$ respectively.  The \emph{adhesion} of the decomposition is 
\begin{equation*}
max_{uv \in E(T)} \left | \delta_G \left ( \bigcup_{t \in V(T_v)} X_t \right)\right |
\end{equation*}
when $T$ has at least one edge, and 0 otherwise.  
The sets $\{X_t: t \in V(T)\}$ are called the \emph{bags} of the decomposition.
\end{DEF}
In the definition, we allow bags to be empty.  

Let $G$ be a graph and $(T, \zX)$ a tree-cut decomposition of $G$.  Let $t \in V(T)$ be a vertex of $T$.  The \emph{torso of $G$ at $t$} is the graph $H$ defined as follows.  If $|V(T)| = 1$, then the torso of $G$ at $t$ is simply itself.  If $|V(T)| \ge 2$, let the components of $T-t$ be $T_1, \dots, T_l$ for some positive integer $l$.  Let $Z_i = \bigcup_{x \in V(T_i)} X_x$ for $1 \le i \le l$.  Then $H$ is made from $G$ by consolidating each set $Z_i$ to a single vertex $z_i$.  The vertices $X_t$ are called the \emph{core vertices} of the torso.  The vertices $z_i$ are called the \emph{peripheral vertices} of the torso.

We can now state the structure theorem for excluded immersions.  A graph has \emph{$(a, b)$-bounded degree} if there are at most $a$ vertices with degree at least $b$.

\begin{theorem}[\cite{W1}]\label{thm:weakdecomp2}
Let $G$ be a graph and $t\ge 1$ a positive integer.  If $G$ does not admit $K_t$ as a weak immersion, then there exists a tree-cut decomposition $(T, \zX)$ of $G$ of adhesion less than $t^2$ such that each torso has $(t,t^2)$-bounded degree.
\end{theorem}

Tree decompositions and their corresponding tree width were introduced by Halin \cite{H} and independently by Robertson and Seymour \cite{RS2}.  The parameters have proven immensely useful in structural graph theory.  Several of the results going forth will use the parameter tree-width.  However, as we will not use any specific properties of tree decompositions in this article, we omit the technical definitions here.  See \cite{Diestel} for further background on tree-width.  

Given tree-cut decompositions, it is natural to ask what is an appropriate definition for the width of such a decomposition. While some care must be taken to deal with 1 and 2 edge cuts, in 3-edge connected graphs, the width is the maximum of the adhesion and the size of the torsos.  See \cite{W1} for more details.  As we will only be considering tree-cut decompositions of graphs which are at least 3-edge connected, we use a simplified definition here for the width of a tree-cut decomposition.
 \begin{DEF}
Let $G$ be a $3$-edge connected graph and $(T, \zX)$ a tree-cut decomposition of $G$.  For each vertex $t \in V(T)$, let $X_t$ be the bag at the vertex $t$.   Let $H_t$ be the torso of $G$ at $t$.  Let $\alpha$ be the adhesion of the decomposition.  The \emph{width} of the decomposition is 
\begin{equation*}
max \{\alpha\} \cup \{|V(H_t)|: t \in V(T)\}.
\end{equation*}
The \emph{tree-cut width} of the graph $G$, also written $tcw(G)$, is the minimum width of a tree-cut decomposition.
\end{DEF} 

 

Note that the tree-cut width of a graph $G$ is not the same as the tree-width of the line graph of $G$.  For example, under the full definition for tree-cut width taking into account one and two edge connected graphs, it holds that all trees have tree-cut width one.   Trivially, there exist trees with vertices of arbitrarily large degree; therefore their line graphs contain arbitrarily large clique subgraphs.  We conclude that the line graphs of trees can have arbitrarily large tree-width.

Tree-cut decompositions share many of the natural properties of tree decompositions. See \cite{W1} for further details.  One fact which we will use in the sections to come is the following result.

\begin{LE}[\cite{W1}]\label{lem:bdedtw}
Let $w, d \ge 1$ be positive integers and let $G$ be a graph with $\Delta(G) \le d$ and  tree-width at most $w$. Then there exists a tree-cut decomposition of adhesion at most $(2w+2)d$ such that every torso has at most $(d+1)(w+1)$ vertices.  Specifically, the tree-cut width of $G$ is at most $(2w+2)d$.
\end{LE}

Thus, if a graph has bounded tree-width and bounded degree, then it has bounded tree-cut width.  However, the converse is not true.  Again, trees have tree-cut width one but clearly can have arbitrarily large degree.  Also, if we consider the graph consisting of two vertices with $t$ parallel edges connected them, then the graph is $t$-edge connected but has tree-cut width two.  However, as a consequence of our proof of Theorem \ref{thm:main}, we will see that if we eliminate these two possibilities, then the converse of Lemma \ref{lem:bdedtw} does hold.  This is proven in Section \ref{sec:bdeddeg}.

\section{Tangles, minors, and spiders}\label{sec:tangles}

In this section, we introduce many of the graph minors tools which we will use, including tangles, and see how they show a packing result for subgraphs called spiders. Let $G$ be a graph, $k\ge1$ a positive integer, and $X \subseteq V(G)$.  An \emph{$X$-spider of order $k$} consists of $k$ pairwise edge-disjoint paths $P_1, \dots, P_k$ and a vertex $v \in V(G) \setminus X$ such that each $P_i$ has one endpoint equal to $v$, the other endpoint in $X$, and no internal vertex in $X$.  The vertex $v$ is called the \emph{body} of the spider.

Finding and packing spiders can be thought of as a first step of
finding immersions: a spider of order $k$ can possibly serve as the
image of a vertex of degree $k$ and the incident edges. Therefore, it
will be useful to know that spiders have the Erd\H os-P\'osa property: either there is a large edge-disjoint collection or there
is a bounded-size set of edges hitting all spiders.  
\begin{theorem}\label{thm:EPspiders}
There exists a function $f(t,k)$ satisfying the following.  Let $G$ be a graph, $X \subseteq V(G)$, and $k, t \ge 1$ a positive integer.  Either $G$ has $t$ pairwise edge-disjoint $X$-spiders each of order $k$, or there exists a set $Z$ of at most $f(t,k)$ edges intersecting every $X$-spider of order $k$ in $G$.
\end{theorem}
Note that the edge-disjoint spiders in Theorem \ref{thm:EPspiders} may share body vertices.  Theorem \ref{thm:EPspiders} is implicit in \cite{RS23} using the language of tangles.  We give the proof below after presenting the necessary background on separations and tangles.

A \emph{separation} in a graph $G$ is an ordered pair of subgraphs $(A, B)$ which are pairwise edge disjoint such that $A \cup B = G$.  The order of the separation is $|V(A) \cap V(B)|$.  The separation is trivial if $A$ is a subgraph of $B$ or vis versa, $B$ a subgraph of $A$.  Note that separations are usually defined to be unordered pairs.  However,  since  we will only consider separations in the context of tangles where the separations are necessarily ordered, it will be convenient for us to always assume that separations are ordered.  

Tangles play an important part of the theory of minors and allow one to disregard ``small" pieces of the graph which are separated off by ``small" cutsets.  If $G$ is a graph and $\Theta$ a positive integer, a \emph{tangle} in $G$ of order $\Theta$ is a set $\zT$ of separations of $G$, each of order $< \Theta$, such that
\begin{itemize}
\item[i.] for every separation $(A, B)$ of $G$ of order $< \Theta$, one of $(A, B)$ or $(B, A)$ is in $\zT$, and
\item[ii.] if $(A_1, B_1)$, $(A_2, B_2)$, $(A_3, B_3) \in \zT$ then $A_1 \cup A_2 \cup A_3 \neq G$, and
\item[iii.] if $(A, B)\in  \zT$, then $V(A) \neq V(G)$.  
\end{itemize}
See \cite{RS10} for a more in depth introduction to tangles.

A graph $G$ contains a graph $H$ as a minor if $H$ can be obtained from a subgraph of $G$ by repeatedly contracting edges.  If $H$ is a simple graph, then a \emph{model} of an $H$-minor in $G$ is a set of subsets of $V(G)$ $\{X_v \subseteq V(G): v \in V(H)\}$ such that
\begin{itemize}
\item[i.] for all $u, v \in V(G)$, $u \neq v$, $X_u \cap X_v = \emptyset$, 
\item[ii.] $G[X_v]$ is connected for all $v \in V(H)$, and
\item[iii.] for all $uv \in E(H)$, there exists an edge of $G$ with one end in $X_u$ and one end in $X_v$.
\end{itemize} 
The sets $X_v$ are called the \emph{branch sets} of the model.  Clique minors give rise in a natural way to large tangles in a graph.  Let $G$ be a graph, let $t$ and $k$ positive integers such that $t \ge \frac{3}{2} k$, and let $\{X_i: 1 \le i \le t\}$ be the branch sets of a model of $K_t$ in $G$.  For every separation $(A, B)$ of order less than $k$, exactly one of $V(A) \setminus V(B)$ or $V(B) \setminus V(A)$ contains a branch set $X_i$ for some $i$.  Let $\zT$ be the set of separations $(A, B)$ of order less than $k$ such that $V(B) \setminus V(A)$ contains a branch set $X_i$ for some $i$.  Then $\zT$ forms a tangle of order $k$.  We refer to the tangle $\zT$ as the tangle \emph{induced} by the model of $K_t$ of order $k$.  Note that the requirement $t \ge \frac{3}{2}k$ (and not simply $t > k$) is necessary to ensure that property ii in the definition of a tangle holds.

We will need the following theorem of Robertson and Seymour (\cite{RS23}, Proposition 7.2).  Given a tangle $\zT$ in a graph $G$ of order $\Theta$, a set $X \subseteq V(G)$ is \emph{free} with respect to $\zT$ if there does not exist $(A, B) \in \zT$ of order strictly less than $|X|$ such that $X \subseteq V(A)$.

\begin{theorem}[\cite{RS23}]\label{thm:RStangles}
Let $\zT$ be a tangle in a graph $G$, and let $W \subseteq V(G)$ be free relative to $\zT$ with $|W| \le w$.  Let $h\ge 1$ be an integer, and let $\zT$ have order $\ge(w+h)^{h+1} + h$.  Then there exists $W' \subseteq V(G)$ with $W \subseteq W'$ and $|W'| \le (w+h)^{h+1}$ such that for every $(C, D) \in \zT$ of order $< |W| + h$ with $W \subseteq V(C)$, there exists $(A', B') \in \zT$ with $W' \subseteq V(A' \cap B'), |V(A' \cap B') \setminus W'| < h$, $C \subseteq A'$ and $E(B') \subseteq E(D)$.
\end{theorem}

We will use a slightly reformulated statement which follows
immediately from Theorem \ref{thm:RStangles}.  Given a tangle $\zT$ in
a graph $G$ and a set $Z \subseteq V(G)$ such that $|Z|$ is less than
the order of $\zT$, the tangle $\zT - Z$ is defined as the set of
separations $(A', B')$ of $G-Z$ such that there exists a separation
$(A, B) \in \zT$ such that $Z \subseteq V(A \cap B)$, $A -
Z=A'$, and $B-  Z=B'$ hold. Robertson and Seymour proved that
$\zT- Z$ is indeed a tangle \cite{RS10}.


\begin{theorem}\label{thm:imp2}
Let $G$ be a graph and $\zT$ a tangle in $G$ of order $t$.  Let $k$ and $w$ be positive integers with $t \ge (k+w)^{k+1} + k$.  Let $\{X_j \subseteq V(G): j \in J\}$ be a family of subsets of $V(G)$ indexed by some set $J$ with $|X_j| = k$ for all $j \in J$.  Then there exists a set $J' \subseteq J$ satisfying the following.
\begin{enumerate}
\item for all $j, j' \in J'$, $X_j \cap X_{j'} = \emptyset$, and 
\item $X = \bigcup_{j \in J'} X_j$ is free.  
\end{enumerate}
Moreover, if $|X| < w$, then there exists a set $Z$ with $X \subseteq Z$ and $|Z| \le (w+k)^{k+1}$ satisfying the following: 
\begin{enumerate}
\item[3.] for all $j \in J$, either $X_j \cap Z \neq \emptyset$ or $X_j$ is not free in $\zT - Z$.
\end{enumerate}
\end{theorem}

\begin{proof}
Pick $J' \subseteq J$ such that $J'$ satisfies 1.~and 2.  Moreover, pick $J'$ to maximize $|X|$ for $X = \bigcup_{j \in J'} X_j$.  If $|X| < w$, we apply Theorem \ref{thm:RStangles} to $W = X$ with $h = k$ and the value $w$.  Let $W'$ be the set given by Theorem \ref{thm:RStangles}.  Let $j \in J$ and consider $X_j$.  Assume $X_j \cap W' = \emptyset$.  By the choice of $J'$ to maximize $|X|$, $X_j \cup X$ is not free in $\zT$.  Thus, there exists a separation $(C, D) \in \zT$ with $X_j \cup X \subseteq C$ and of order $< |X| + k$.  The separation $(A', B')$ guaranteed by Theorem \ref{thm:RStangles} ensures that the set $X_j$ is not free in $\zT - W'$, as desired.
\end{proof}
Note that the set $J'$ in Theorem \ref{thm:imp2} may be empty, but in this case it must hold that every $X_i$ is not free in $\zT$.   
%
%
The proof of Theorem~\ref{thm:EPspiders} follows easily by invoking
Theorem~\ref{thm:imp2} on the line graph of $G$:
\begin{proof}[Proof (of Theorem~\ref{thm:EPspiders})]
  First consider the following observation.  If $e_1, e_2, \dots, e_k$
  are distinct edges sharing a common endpoint $v$, then they form a
  $K_k$ subgraph, call it $K$, in the line graph of $G$.  If there
  exist $k$ vertex disjoint paths $P_1, \dots, P_k$ from $K$ to
  $\delta(X)$ in the line graph of $G$, then in the original graph
  $G$, those paths will contain an $X$-spider of order $k$ with body
  equal to $v$.  There is a subtlety here, in that the $X$-spider may
  not contain the edges $e_1, \dots, e_k$ since some $P_i$ may contain
  two edges incident the vertex $v$.

  To find the pairwise disjoint spiders or the bounded hitting set,
  first consolidate the vertex set $X$ to a single vertex $x$.  In the
  line graph, the set of edges $\delta(x)$ forms a large clique
  subgraph which induces a large order tangle.  We let the subsets
  $X_j$ be all possible subsets of $k$ edges with a common endpoint in
  $V(G) \setminus X$ and apply Theorem \ref{thm:imp2} to these sets
  $X_j$ in the line graph.  Either we find $t$ of them whose union is
  free in the line graph, corresponding to $t$ pairwise edge-disjoint
  spiders in $G$, or alternatively, we get the bounded size hitting
  set as desired.
\end{proof}


\section{Excluding a spider}\label{sec:spider}

In this section, we give an exact characterization of when a given graph has an $X$-spider for a fixed subset $X$ of vertices.  However, in the applications to come, we will need a stronger version of this theorem and so we generalize the statement in terms of tangles.

In order to use the results of the previous section, which are based
on tangles and vertex separations, we will need to pass back and forth
between the graph $G$ where we are looking for a spider and the line
graph of $G$.  This leads us to make the following definition.  Let
$G$ be a graph and $U \subseteq V(G)$.  We denote by $N(U)$ the set of
vertices of $V(G) \setminus U$ with at least one neighbor in $U$.  The
set $U$ defines an edge cut in $G$, namely $\delta(U)$.  This edge cut
in $G$ corresponds to a separation in the line graph.  We define the
separation $(A, B)$ of the line graph as follows. Let $L$ be the
line graph $L(G)$.  Let $A = L[E(G[U]) \cup \delta_G(U)]$.  Let $B =
L[E(G - U) \cup \delta_G(U)] - E(L[\delta_G(U)])$. We refer to
$(A, B)$ as the \emph{canonical separation in $L(G)$} for $U$.  Note
that the order of the canonical separation is $|\delta_G(U)|$.

Define a \emph{$k$-star} in a graph $G$ to be a set $F$ of $k$ edges for which there exists a vertex $u$ such that every edge in $F$ has $u$ as an endpoint.  The vertex $u$ is called the \emph{center} of the star.  We now characterize when a graph has a $k$-star which is free with respect to the given tangle in the line graph.  We first need two easy claims on properties of tangles.  The first follows from property ii in the definition of a tangle and the second follows from the first, again along with property ii in the definition of a tangle.

\begin{OB}\label{obs:tangleprops}
  Let $G$ be a graph and $\zT$ a tangle in $G$.  Let $(A, B) \in \zT$ and let $S=V(A)\cap V(B)$.
  Let $(\bar{A}, \bar{B})$ be the separation with $\bar{A} = A \cup
  G[S]$ and $\bar{B} = B - E(G[S])$.  Then $(\bar{A},
  \bar{B}) \in \zT$.  Let $(A', B')$ be a separation with $|V(A \cap
  B)| \ge |V(A' \cap B')|$ such that $V(A') \subseteq V(A)$.  Then
  $(A', B') \in \zT$.
\end{OB}

\begin{LE}\label{lem:nokspidertangle}
Let $G$ be a graph, and let $L(G)$ be the line graph of $G$.   Let $t, k$ be positive integers with $k <t$.  Let $\zT$ be a tangle in $L(G)$ of order $t$.  Let $U \subseteq V(G)$.  There does not exist a $k$-star $F$ with center $u \in U$ which is free in $\zT$ if and only if there exists a positive integer $l$ and subsets $U_1, \dots, U_l \subseteq V(G)$ such that:
\begin{enumerate}
\item $U_i \cap U_j = \emptyset$ for $1 \le i < j \le l$ and $U \subseteq \bigcup_1^l U_i$, 
\item $|\delta_G(U_i)| < k$ for all $1\le i \le l$, and 
\item if $(A_i, B_i)$ is the canonical separation in $L(G)$ for $U_i$, then $(A_i, B_i) \in \zT$ for all $1 \le i \le l$.
\end{enumerate}
\end{LE}

\begin{proof}
To see necessity, assume we have such sets $U_1, \dots, U_l$ satisfying 1.-3.  For any $k$-star $F$ with center  $u \in U$, there exists an index $i$ such that $u \in U_i$.  But then the canonical separation $(A_i, B_i)$ for $U_i$ satisfies $F \subseteq V(A_i)$ and is of order at most $k-1$.  Thus, $F$ is not free.

We now show sufficiency.  Assume the statement is false, and pick a counterexample $G$, $\zT$, $U$.  Let $L(G)$ be the line graph of $G$.   Assume there does not exist a $k$-star $F$ which is free with respect to $\zT$ in $L(G)$.  Let $\{F_i: i \in I\}$ be the set of all possible distinct $k$-stars in $G$ with center vertex in $U$ indexed by a set $I$.  Let $u_i$ be the center vertex of $F_i$ for $i \in I$.

In general, separations in the line graph $L(G)$ do not immediately correspond to edge cuts in $G$: one must take into account trivial separations and separations which are not minimal.

\begin{Claim}\label{cl:1}
  Let $F$ be a $k$-star in $G$ with center $u \in U$.  Let $(A, B) \in
  \zT$ be a separation of order strictly less than $k$ in $L(G)$ such
  that $F \subseteq E(A)$ and $E(A \cap B) \subseteq E(A)$.  Assume
  further that $(A,B)$ is selected from all such sets to have minimum
  order. Then there exists a set $W \subseteq V(G)$ with $u \in W$
  such that $(A, B)$ is the canonical separation for $W$.
\end{Claim}
\begin{cproof}
Observe that $V(A) \setminus V(B)$ is not empty, as $|V(A)| \ge k$, and $V(B) \setminus V(A) \neq \emptyset$ by the properties of a tangle.  Thus, $L(G) - V(A \cap B)$ is disconnected, which implies that $G - V(A \cap B)$ is also disconnected.  Each component of $G - V(A \cap B)$ intersects at most one of $V(A)\setminus V(B)$ and $V(B) \setminus V(A)$.  Let $H$ be the components of $G-V(A \cap B)$ containing edges of $V(A) \setminus V(B)$; thus $V(A) \setminus V(B) = E(H)$.  Let $W = V(H)$.  Note that the center $u$ of $F$ is contained in $W$.  We claim that $(A, B)$ is the canonical separation of $W$.  To see this, it suffices to show that every edge of $V(A \cap B)$ has one end in $W$ and one end in $V(G) \setminus W$.  This follows from our choice of $(A, B)$ to be a separation of minimal order in $\zT$ with $F \subseteq V(A)$, proving the claim.
\end{cproof}

We fix a set $\{W_j \subseteq V(G): j \in J\}$ of subsets of $V(G)$ such that if $(A_j, B_j)$ is the canonical separation for $W_j$ for $j \in J$, then 
\begin{enumerate}
\item[a.] $(A_j, B_j) \in \zT$ for all $j \in J$, 
\item[b.] for all $i \in I$, there exists $j \in J$ such that $u_i \in W_j$, and 
\item[c.] subject to a and b, $\sum_{j \in J} |W_j|$ is minimized.
\end{enumerate}

Note that such a set $\{W_j \subseteq V(G): j \in J\}$ exists by Claim \ref{cl:1} and the fact that for every $k$-star $F$ with center in $U$, we can find a separation $(A, B) \in \zT$ of $L(G)$ with $F \subseteq V(A)$ of order strictly less than $k$.  

We claim that the set $\{W_i: i \in J\}$ are pairwise disjoint.  Assume that there exist $j, j' \in J$, $j \neq j'$, such that $W_j \cap W_{j'} \neq \emptyset$.  If $W_j \subseteq W_{j'}$, then $E(A_j) \subseteq E(A_{j'})$, and therefore $\{W_i: i \in J - j\}$ satisfy a and b, contrary to our choice of $\{W_i: i \in J\}$.  Thus, we may assume that both $W_j \setminus W_{j'}$ and $W_{j'} \setminus W_j$ are non-empty.  It is a basic property of edge cuts in graphs that 
\begin{equation*}
|\delta(W_j)| + | \delta(W_{j'})| \ge |\delta(W_j \setminus W_{j'})| +  |\delta(W_{j'} \setminus W_{j})|.
\end{equation*}
Thus, without loss of generality, we may assume that $|\delta(W_j \setminus W_{j'})|< k$.  Then if $(A', B')$ is the canonical separation of $ W_j \setminus W_{j'}$, by Observation \ref{obs:tangleprops}, we have that $(A', B') \in \zT$.  Thus, $\{W_i: i \in J, i \neq j\} \cup \{W_j \setminus W_{j'}\}$ satisfies a and b, contrary to our choice to minimize $\sum_{i \in J} |W_i|$.  

The sets $\{W_i: i \in J\}$ are pairwise disjoint.  It is possible that they do not satisfy $\bigcup_{i \in J} W_i  \supseteq U$, however in this case, it must be true that for every $v \in U \setminus \bigcup_{i \in J} W_i$, we have that $|\delta(v)| < k$ by our choice $\{W_i: i \in J\}$ to contain the center of every $k$-star with center in $U$.  Thus, $\{W_i: i \in J\} \cup \{\{v\}: v \in U \setminus \{W_i: i \in J\}\}$ satisfies the statement of the lemma.  This completes the proof.
\end{proof}

From Lemma \ref{lem:nokspidertangle}, it is easy to characterize when a given graph has an $X$-spider of order $k$. 
\begin{LE}\label{lem:nokspider}
Let $G$ be a graph, $X \subseteq V(G)$, and $k \ge 1$ a positive integer.  Assume $|\delta(X)| \ge \frac{3}{2}k$.  Then $G$ contains an $X$-spider of order $k$ if and only if there does not exist a partition $U_1, \dots, U_m$ of $V(G) \setminus X$ such that for all $1 \le i \le m$, $|\delta(U_i)| < k$.
\end{LE}
\begin{proof}
Necessity is immediate.  In order to see sufficiency, assume $G$ does not contain an $X$-spider of order $k$.  Let $G'$ be obtained from $G$ by consolidating the vertex set $X$ into a single vertex $x$.  Clearly, there exists an $X$-spider in $G$ of order $k$ if and only if there exists an $x$-spider of order $k$ in $G'$.  The set of edges $\delta_{G'}(x)$ forms a clique subgraph $H$ of $L(G')$ of order at least $\frac{3}{2}k$.  Let $\zT$ be the tangle of order $k$ induced by this model of a clique minor in $G'$.  

We apply Lemma \ref{lem:nokspidertangle} to $G'$ with the set $U = V(G') - x$ and the tangle $\zT$.  Assume, to reach a contradiction, that there exists a $k$-star $F$ which is free with respect to $\zT$ and has center $u \in U$.  In $L(G')$ there does not exist a separation of order $<k$ separating $F$ from $V(H)$.  Thus, there exist $k$ pairwise vertex disjoint paths from $F$ to $V(H) = \delta_{G'}(x)$ in $L(G')$, and consequently, there would exist an $\{x\}$-spider of order $k$ in $G'$, a contradiction.  

Thus, by Lemma \ref{lem:nokspidertangle}, there exist sets $U_1, \dots, U_l$ in $V(G')$ satisfying $1.-3.$  Observe that no $U_i$ contains an edge incident to $x$ by the definition of the tangle $\zT$, and consequently, no $U_i$ contains the vertex $x$.  We conclude that $U_1, \dots, U_l$ is a partition of $V(G') - x = V(G) \setminus X$.  The lemma now follows from the observation that $\delta_{G'}(U_i) = \delta_G(U_i)$ for all $1 \le i \le l$.
\end{proof}
Note that the proof of Lemma \ref{lem:nokspidertangle} can be replicated to show Lemma \ref{lem:nokspider} eliminating the assumption that $|\delta(X)| \ge \frac{3}{2}k$.  
\section{Bounded degree}\label{sec:bdeddeg}

In this section, we show that a $k$-edge connected graph contains a vertex with a sufficiently large neighborhood, then we can immerse any graph of maximum degree $k$.  One consequence of this is to characterize when the converse of Lemma \ref{lem:bdedtw} holds.  Dvorak and Klimosova \cite{DK} have recently found a proof of this statement with better bounds and which lends itself to finding strong immersions as opposed to weak immersions.  We include our proof here, as it illustrates some of the ideas which we will use in the proof of Theorem \ref{thm:main}.  

We first define the graph $S_{l, m}$ to be the graph with $m+1$ vertices $x_1, \dots, x_m, y$ and $l$ parallel edges from $x_i$ to $y$ for all $1 \le i \le m$.  See Figure \ref{fig2}.
\begin{figure}[htb]\label{fig2}
\begin{center}
\includegraphics[scale = .5]{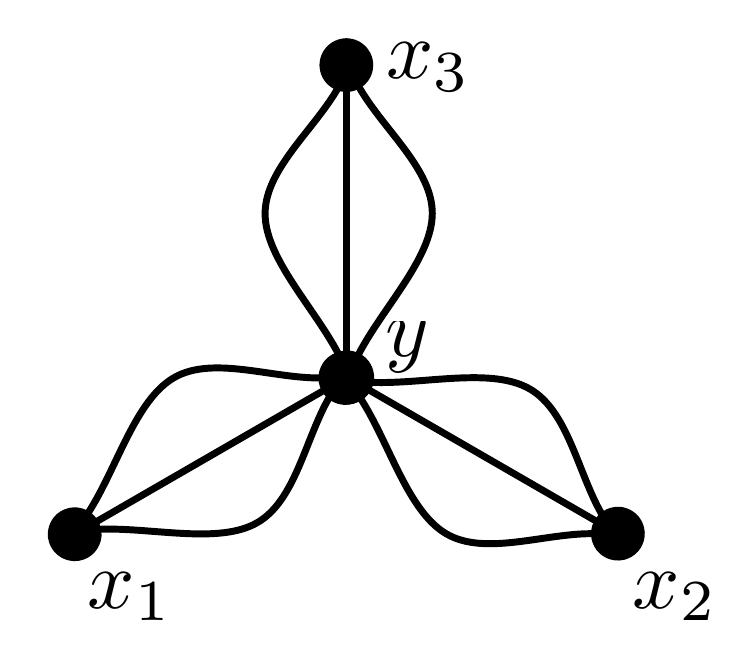}
 \caption{The graph $S_{3,3}$.}
 \end{center}
\end{figure}
The graphs $S_{l, m}$ have the useful property that they contain any fixed graph $H$ as a strong immersion for appropriate chosen values $l$ and $m$.  Specifically, let $H$ be a graph of maximum degree at most $k$ on $n$ vertices for positive integers $k$ and $n$.  Then the graph $S_{k, n}$ contains $H$ as a strong immersion.  To see this, consider the following.  Given the graph $H$, subdivide each edge of $H$, and identify all the new vertices of degree two to a single vertex.  The resulting graph is a subgraph of $S_{k, n}$ for $n = |V(G)|$ and $k$ equal to the maximum degree of $H$.  Reversing this process shows how to arrive at $H$ by identifying pairwise edge disjoint paths linking the branch vertices.

\begin{LE}\label{lem:bdeddeg}
Let $k \ge 1$ be a positive integer, let $G$ be a $k$-edge connected graph, and let $H$ be a graph of max degree $k$ on $n$ vertices.  Let $f$ be the function in Theorem \ref{thm:EPspiders}.  If $G$ has no $H$-immersion, then for all $v \in V(G)$
\begin{equation*}
|N(v)| \le (2k) f\left(k,  k^2\binom{n-1}{k-1}\right) + n.\end{equation*} 
\end{LE}
\begin{proof}
Assume $G$ does not admit an immersion of $H$.  Thus, $G$ has no immersion of $K_n$ as well.  By the structure theorem (Theorem~\ref{thm:weakdecomp2}), there exists a tree-cut decomposition $(T, \zX)$ of $G$ having adhesion less than $n^2$ such that for every $t \in V(T)$, the torso $H_t$ of $G$ at the vertex $t$ has $(n, n^2)$-bounded degree.  Assume, to reach a contradiction, that $G$ has a vertex $z$ which satisfies $|N(z)| \ge  (2k) f\left(k,  k^2\binom{n-1}{k-1} \right) + n$, and fix $t$ to be the vertex of $T$ such that $z \in X_t$.  Let $H_t$ be the torso of $G$ at $t$, and let $Z$ be the set of vertices of $H_t$ of degree at least $n^2$.  Note $z \in Z$.  

We apply Theorem \ref{thm:EPspiders} to either find many disjoint
$Z$-spiders in $G$ of order $k$ or a bounded size-hitting set.
Assume, as a case, that there exist $m$ pairwise edge-disjoint
$Z$-spiders $S_1, \dots, S_m$ for $m = k^2\binom{n-1}{k-1} =
\frac{k}{n^2} k\binom{n}{k}(kn)$.  Note that these spiders may share
body vertices, but we argue that there is a large subset of them with
pairwise distinct body vertices.  Every vertex $v \in V(G)\setminus Z$
either has $\deg_G(v) < n^2$ or $v \notin X_t$.  Consider a vertex $v
\notin X_t$.  There exists a subset $U \subseteq V(G) \setminus Z$
with $v \in U$ such that $|\delta_G(U)| \le n^2$ (as the adhesion of $(T,\zX)$ is less than $n^2$), and consequently
there are at most $n^2/k$ distinct indices $i$ such that $S_i$ has the
vertex $v$ as a center.  If instead we consider $v \in X_t \setminus
Z$, $\deg_G(v) \le n^2$, and again there are at most $n^2/k$ distinct
indices $i$ such that $S_i$ has $v$ as a body by the bound on the
degree of $v$.  We conclude that there exists a subset $I \subseteq
\{1, \dots , m\}$ of size at least $k\binom{n}{k}(kn)$ that
corresponds to spiders with pairwise distinct bodies. That is, if
$v_i$ is the body of $S_i$ for $i \in I$, then for all $i, j \in I$,
$i \neq j$, $v_i \neq v_j$.

We may assume that each $Z$-spider of order $k$ in our collection
contains at most $k$ vertices of $Z$ and each path in a spider
contains exactly one vertex of $Z$.  There are at most $k
\binom{n}{k}$ different subsets of $Z$ of size at most $k$.  Thus,
there exists a subset $Z'\subseteq Z$ of size at most $k$ and a subset
$I' \subseteq I$ with $|I'|\ge kn$ such that for every $i \in I'$, $Z'
= V(S_i) \cap Z$.  Pick $n$ distinct indices in $I'$ and let $x_1$,
$\dots$, $x_n$ be the corresponding body vertices. Let $z'$ be an
arbitrary vertex of $Z'$. We construct an immersion of $S_{k,n}$ by
finding $k$ paths from each $x_j$ to $z'$ such that these $kn$ paths
are pairwise edge disjoint. We start with the $k$ edge-disjoint paths
of the spider at $x_j$. If one of these paths terminates at a vertex
$z''\in Z'$ different from $z'$, then we extend the path by picking a
spider $S_i$, $i\in I'$ not yet used, using one of the paths of $S_i$
to go from $z''$ to the body of $S_i$, and then using one of the paths
of $S_i$ going from the body of $S_i$ to $z'$. (By definition of the
set $I'$, spider $S_i$ does have paths terminating in $z'$ and in
$z''$.) Taking into account that the spider at $x_j$ has a path going
directly to $z'$, we have to repeat this rerouting argument at most
$(k-1)n$ times. Thus $|I'|\ge kn$ implies that we can pick a different
spider $S_i$ each time.  We conclude that $G$ contains $S_{k,n}$ as an
immersion, and consequently contains $H$ as an immersion, a
contradiction. This concludes the case when there are many pairwise
edge-disjoint spiders.

Thus, we may assume from Theorem \ref{thm:EPspiders} that there exists a set $F \subseteq E(G)$ of size at most $f\left(k,   k^2\binom{n-1}{k-1} \right)$ such that $G-F$ does not contain a $Z$-spider of order $k$.  By Lemma \ref{lem:nokspider}, there exists a partition $Y_1, \dots, Y_p$ of $V(G-F) - Z= V(G)-Z$ such that $|\delta_{G-F}(Y_i)|<k$ for some positive integer $p$.  

In the graph $G-F$, the vertex $z$ has $|N_{G-F}(v)| \ge |N_{G}(v)| - |F| \ge (2k-1) f\left(k,   k^2\binom{n-1}{k-1}\right) + n$.  Thus, $z$ has at least $(2k-1) f\left(k,   k^2\binom{n-1}{k-1}\right)$ distinct neighbors in $V(G) \setminus Z$ in the graph $G-F$.  As $z$ has at most $k-1$ neighbors in each set $Y_i$, we conclude that $p > 2 f\left(k,   k^2\binom{n-1}{k-1}\right) \ge 2 |F|$.  However, every set $Y_i$ must contain an endpoint of some edge in $F$ by the edge-connectivity of $G$.  This final contradiction completes the proof.
\end{proof}

Recall that Lemma \ref{lem:bdedtw} shows that any graph with both bounded tree-width and bounded degree also has bounded tree-cut width.  As we observed in Section \ref{sec:tcd}, the converse is not true.  However, if we eliminate parallel edges and assume a minimal amount of edge connectivity, the converse does hold.  The proof uses the Grid theorem of Robertson and Seymour.  As we will not need these concepts further in the article, we direct the reader to \cite{Diestel} for additional details.

\begin{theorem}\label{thm:main1}
Let $G$ be a 3-edge connected simple graph.  For all $k \ge 1$, there exists a value $D = D(k)$ such that every graph with tree-cut width at most $k$ has maximum degree at most $D$ and tree-width at most $D$.  
\end{theorem}

\begin{proof}
Let $G$ be a graph of tree-cut width at most $k$ for some fixed $k \ge 1$.  Thus, $G$ does not admit an immersion of a large wall \cite{W1}.  As a consequence, the tree-width of $G$ must also be bounded, as sufficiently large tree-width will ensure the existence of a large wall subdivision which forms a large wall immersion. 

Similarly, by Lemma \ref{lem:bdeddeg}, we we know that $|N(v)|$ is bounded for all $v \in V(G)$.  Given that the graph is simple, it follows that the maximum degree of $G$ is bounded as well.
\end{proof}

\section{Proof of Theorem \ref{thm:main}}\label{sec:thm1}

In order to prove Theorem \ref{thm:main}, we will make use of several auxiliary results.  The first has been shown by Chudnovsky, Dvorak, Kilmosova and Seymour \cite{CDKS}.
\begin{theorem}[\cite{CDKS}]\label{thm:CDKS}
Let $G$ be a 4-edge connected graph.  Then for all $t$ there exists a $W$ such that either $G$ has tree-width at most $W$ or the line graph of $G$ contains a $K_t$-minor. 
\end{theorem}

The second result we will use is due to Robertson and Seymour \cite{RS9}.

\begin{theorem}[\cite{RS9}]\label{thm:RS2}
Let $T = \{s_1, \dots, s_k, t_1, \dots, t_k\}$ be a set of $2k$ distinct vertices in a graph $G$.  Let $X_1, \dots, X_{3k}$ form a model of a $K_{3k}$ minor.  Assume there does not exist a separation $(A, B)$ of order $< 2k$ and an index $i$ such that $T \subseteq A$ and $X_i \subseteq B \setminus A$.  Then there exist $k$ pairwise vertex disjoint paths $P_1, \dots, P_k$ such that the endpoints of $P_i$ are $s_i$ and $t_i$. 
\end{theorem}

The proof of Theorem \ref{thm:main} uses the fact that Lemma \ref{lem:bdeddeg} essentially bounds the degree of a potential counterexample.  However, Lemma \ref{lem:bdeddeg} only bounds the size of the neighborhood of each vertex; to bound the degree, we need to eliminate the possibility of large numbers of parallel edges between pairs of vertices.  This is a somewhat annoying technicality.  Thus, we effectively split the proof of Theorem \ref{thm:main} into two parts.  The first, stated below as Theorem \ref{thm:main3} restricts to the case when there are a bounded number of parallel edges between any pair of vertices.  Then in order to prove Theorem \ref{thm:main}, we only need to bound the number of such parallel edges.

\begin{theorem}\label{thm:main3}
There exists a function $g'$ satisfying the following.  Let $k \ge 4, n \ge 1, D \ge 1$ be positive integers.  
 Let $H$ be a graph with maximum degree $k$ on $n$ vertices.  Let $G$ be a $k$-edge connected graph such that $deg(v) \le D$ for all $v \in V(G)$.  Then either $G$ admits an immersion of $H$, or $G$ has tree-cut width at most $g'(k,n,D)$.
\end{theorem}

\begin{proof}
Let $G$ be $k$-edge connected but not admit an immersion of $H$.  Let $|V(H)| = n$.   Let $L(G)$ be the line graph of $G$.  Fix 
\begin{equation*}
\ell = \frac{3}{2}(4k+1)(2Dn)^{k+1}.
\end{equation*}
By Theorem \ref{thm:CDKS}, there exists a value $W$ such that either
$G$ has tree-width at most $W$ or the line graph of $G$ contains $K_\ell$
as a minor.  We set $g'(k, n, D) = (2W+2)D$. If the tree-cut with of $G$ is greater than $g'(k,n,D)$, then Lemma
\ref{lem:bdedtw} implies that $G$ has tree-width greater than $W$, and consequently we may assume in the following that $L(G)$
contains $K_\ell$ as a minor.  Fix a model $K$ of $K_\ell$ as a
minor in $L(G)$.  The model $K$ defines a tangle $\zT$ of order
$4k(2Dn)^{k+1} + k$ in $L(G)$.

Let $\{X_i: i \in I\}$ be the set of $k$-stars in $G$.  We apply Theorem \ref{thm:imp2} to the tangle $\zT$ and the sets $\{X_i: i \in I\}$.  Let $I' \subseteq I$ be the subset satisfying 1.~and 2.~in the statement.  Let $X = \bigcup_{i \in I'} X_i$.  Assume, as a case, that $|X| \ge Dn$.  As every vertex of $G$ has degree at most $D$, and the stars $X_i$ for $i \in I'$ are pairwise disjoint, there exists a subset $I'' \subseteq I'$, $|I''| \ge kn$ such that for every $i, j \in I''$, $X_i$ and $X_j$ have distinct center vertices.  Assign each $X_i$ for $i \in I''$ to a distinct vertex of $H$.  Let $X' = \bigcup_{i \in I''} X_i$, and let $m = |E(H)|$.  There is a natural way to label elements of $X'$ by $s_1, t_1, s_2, t_2, \dots, s_{m}, t_{m}$ such that if there exist pairwise vertex disjoint paths $P_1, \dots, P_{m}$ (in $L(G)$) such that the ends of $P_i$ are $s_i$ and $t_i$, then in the original graph $G$ admits an immersion of $H$.  Given that $X' = \bigcup_{i \in I''}$ is free with respect to the tangle $\zT$, by Theorem \ref{thm:RS2}, there exist the desired disjoint paths $P_1, \dots, P_{m}$, and we conclude that $G$ contains $H$ as an immersion.

Thus, we may assume $|X| < Dn$, and consequently, there exists a set $Z \subseteq E(G)$ with $|Z| \le (2Dn)^{k+1}$ satisfying 3.~in the statement of Theorem \ref{thm:imp2}.  There exists a model $K'$ of a $K_{4k(2Dn)^{k+1}}$ minor in  $L(G-Z)$ obtained by simply discarding any branch set of $K$ containing an element of $Z$.  Then $K'$ induces a tangle $\zT'$ in $L(G -Z)$ of order $2k(2Dn)^{k+1}$.  Note that the set of $k$-stars of $G-Z$ is exactly the set $\{X_i: i \in I \text{ and }X_i \cap Z = \emptyset\}$.  Thus, by property 3.~and the definitions of $\zT$ and $\zT'$, we see that there does not exist a $k$-star in $G-Z$ which is free with respect to $\zT'$.  

We conclude by Lemma \ref{lem:nokspidertangle} that there exists a partition $U_1, \dots, U_s$ of $V(G - Z)$ that satisfies $1.-3.$~in the statement of Lemma \ref{lem:nokspidertangle} for some positive integer $s$.  Note that at most $k-1$ distinct branch sets of $K'$ intersect the edge set $E_{G-Z}(U_i) \cup \delta_{G-Z}(U_i)$ for all $1 \le i \le s$ by property 3.  Thus, given the order of the clique minor $S'$, we conclude that $s > 2(2Dn)^{k+1}$.  However, for all $1 \le i \le s$, there exists at least one edge of $Z$ in $\delta_G(U_i)$ by the overall connectivity of the graph.  By our bound on $s$, we see $|Z| > (2Dn)^{k+1}$, a contradiction.  This completes the proof of the theorem.
\end{proof}

We will need several quick observations in order to eliminate the degree bound in the statement of Theorem \ref{thm:main3}.  
\begin{OB}\label{obs1}
Let $G$ and $H$ be graphs such that $G$ does not admit a (strong) immersion of $H$.  Assume there exists $u, v \in V(G)$ with at least $|E(H)|$ parallel edges from $u$ to $v$ in $G$.  Let $G'$ be the graph obtained from $G$ by consolidating the vertex set $\{u,v\}$.  Then $G'$ does not admit a (strong) immersion of $H$.
\end{OB}

To see that Observation \ref{obs1} holds, let $x$ be the vertex of $G'$ corresponding to $\{u,v\}$.  Then if $G'$ had an immersion of $H$, at most $|E(H)|$ distinct composite paths would use an edge of $\delta_{G'}(x)$.  Given the large number of edges in $G$ linking $u$ and $v$, we can easily extend such an immersion in $G'$ to an immersion in $G$.  Note that if the immersion in $G'$ is strong, the immersion we construct in the original graph $G$ will be strong as well.

\begin{OB}\label{obs2}
Let $G$ and $H$ be graphs.  Let $|V(H)| = n$ and $|E(H)| = m$.  Assume $G$ is connected and has at least $n$ vertices.  Assume as well that, for every pair of vertices $u, v \in V(G)$, if $u$ and $v$ are adjacent then there are at least $m$ parallel edges in $G$ with ends $u$ and $v$.  Then $G$ contains $H$ as an immersion.
\end{OB}

Observation \ref{obs2} can be seen by induction on $|E(H)|$.  If $|E(H)| = 0$, there is nothing to prove.  Otherwise, fix an injection from $V(H)$ to $V(G)$ and pick an edge $uv$ in $H$.  Pick a path in $G$ linking the vertices of $G$ corresponding to $u$ and $v$, delete the edges from $G$, and apply induction.

We now give the proof of Theorem \ref{thm:main}.
\begin{proof}[Proof. Theorem \ref{thm:main}]
Assume $G$ does not admit an immersion of $H$.  Let $|V(H)| = n$.  Consider the auxiliary graph $G'$ defined on the vertex set $V(G)$ such that two vertices $u$ and $v$ of $G'$ are adjacent if there exist at least $kn$ parallel edges in $G$ from $u$ to $v$.  Let the components of $G'$ have vertex sets $V_1, V_2, \dots, V_s$ for some positive integer $s$.  By Observation \ref{obs2}, each component of $G'$ has at most $n-1$ vertices.

Let $G_1$ be the graph obtained from $G$ by consolidating the vertex sets $V_i$ for $1 \le i \le s$.  By Observation \ref{obs1}, the graph $G_1$ does not admit an immersion of $H$.  Let the vertices of $G_1$ be $v_1, \dots, v_s$ where each $v_i$ corresponds to the consolidated set $V_i$ of vertices for $1 \le i \le s$.  Moreover, between every pair of vertices $v_1, v_2$, there are at most $kn^3$ parallel edges given that each of $v_1$ and $v_2$ correspond to at most $n-1$ vertices of $G$.  

By Lemma \ref{lem:bdeddeg}, it now follows that there exists a value $D$ depending only on $k$ and $n$ such that $\Delta(G_1) \le D$.  Thus, by Theorem \ref{thm:main3}, we conclude that the tree-cut width of $G_1$ is at most $g'(k, n, D)$.  Let $(T, \zX_1)$ be a tree-cut decomposition of $G_1$ of minimum width.  Let $\zX_1 = \{X^1_t: t \in V(G)\}$.  For all $t \in V(T)$, let $X_t = \bigcup_{i: v_i \in X^1_t} V_i$ and let $\zX = \{X_t: t \in V(T)\}$.  We conclude that $(T, \zX)$ is a tree-cut decomposition of width at most $g(k, n) = ng'(k,n, D)$.  This completes the proof of the theorem.
\end{proof}
\section{Highly connected graphs with no strong immersion of $K_t$}\label{sec:strong}

DeVos et al.~have shown that every simple graph with large minimum degree has a strong immersion of a fixed clique. 

\begin{theorem}[\cite{DDFMMS}]\label{thm:DDFMMS}
Every simple graph with minimum degree at least $200t$ contains a strong immersion of $K_t$.
\end{theorem}

Trivially, this implies that every simple graph which is $200t$-edge connected contains a strong immersion of $K_t$.  However, if we remove the requirement that the graph be simple, the statement no longer holds.  In fact, the graphs $P_{k, n}$ form examples of graphs which can be arbitrarily highly edge connected and do not contain a strong immersion of $K_4$.  

In this section, we will see that long paths of parallel edges essentially form the only highly edge connected graphs which do not strongly immerse large cliques.  Before proceeding with the theorem, we will need several easy results. First, we show that even though $P_{k,n}$ does not contain a strong immersion of even $K_4$, if there is a single vertex that is adjacent to many vertices of $P_{k,n}$, then a strong immersion of $K_k$ appears.

\begin{LE}\label{lem:pathplusvertex}
Let $x_1$, $\dots$, $x_n$ be vertices in a graph $G$ such that there are at least $k$ edges between $x_i$ and $x_{i+1}$ for every $1\le i < n-1$. Let $y$ be a vertex having at least $k^2$ neighbors in $\{x_1,\dots, x_n\}$. Then $G$ contains a strong immersion of $K_k$. 
\end{LE}
\begin{proof}
A strong immersion of $S_{k,k}$ in $G$ can be easily found, where $y$ is the image of the central vertex of $S_{k,k}$ and the images of the other $k$ vertices are on the path. As $S_{k,k}$ contains $K_k$ as a strong immersion, the claim follows.
\end{proof}
We will need a result on graphs which do not contain $K_{1, \ell}$ as a minor.  There has been considerable work on finding extremal functions on the number of edges or number of vertices of degree at least 3 necessary to force such a minor.  See \cite{DS, KW, BZ}.  The next lemma follows as an easy corollary to a result of Zickfeld \cite{Z}, {Theorem 4.7}.  We include a self-contained proof for completeness.
Recall that if $G$ is a graph and $X \subseteq V(G)$, $N(X)$ denotes the set of vertices in $V(G) \setminus X$ with a neighbor in $X$.  

 \begin{LE}\label{lem:k1t}
   Let $G$ be a connected simple graph and let $\ell \ge 2$ be an integer.  Then
   either $G$ contains a $K_{1, \ell}$ minor or there exists a set $X
   \subseteq V(G)$ such that $|X| < 4\ell$ and the following holds.
   There are at most $2\ell$ components of $G-X$ and every component of
   $G-X$ is a path $P$.  Moreover, for every component $P$ of $G-X$,
   $N(X) \cap V(P)$ is a subset of the endpoints of $P$.
 \end{LE}
 
\begin{proof}
Assume $G$ does not contain a $K_{1,\ell}$ minor.  Pick a spanning tree $T$ of $G$ to maximize the number of leaves.  The tree $T$ is a subdivision of a tree $\bar{T}$ which has no vertex of degree two.  Note that $\bar{T}$ is a minor of $G$, and thus $\bar{T}$ has at most $\ell-1$ leaves.  Since every vertex of $\bar{T}$ which is not a leaf has degree at least three, we see that $|V(\bar{T})| \le 2\ell$, and at least half of the edges of $\bar{T}$ are incident to a leaf of $\bar{T}$.  

Let $Y$ be the set of vertices of degree at least three in $T$, and let $X = Y \cup N_T(Y)$.  We claim $X$ has at most 
$4\ell$ vertices.  For every edge $e \in E(\bar{T})$, there exists a corresponding path $P_e$ in $T$ whose ends have degree not equal to two (in $T$) such that every internal vertex of $P_e$ has degree exactly two in $T$.  Each such path $P_e$ can contain at most two vertices of $N_T(Y)$, and if it contains two such vertices, then it must be the case that both endpoints of $P_e$ are in $Y$.  Thus, given the bounds on $|E(\bar{T})|$ and the fact that at least half the edges of $\bar{T}$ are incident to leaves, we conclude that $|N_T(Y)| \le 3 \ell$.  The bound on $|X|$ now follows.

Assume, to reach a contradiction, that there exists an edge $e$ such that $P_e - X$ has an internal vertex $v$ of degree at least three in $G$.  Let $f$ be an edge of $E(G) \setminus E(T)$ incident to $v$.  If $f$ has both endpoints contained in $P_e$, it is easy to see that there exists an edge $e'$ of $P_e$ such that $(T \cup \{f\}) - e'$ has strictly more leaves.  Alternatively, $f$ has an endpoint not contained in $P_e$.  Then there exists a subpath $P'$ of $P_e$ with one end equal to $v$ and the other end contained in $Y$ such that for every edge $e'$ of $P'$, $(T \cup \{f\}) - e'$ is a tree.  Moreover, since $v$ is an internal vertex of $P_e - X$, we see that $P'$ must have length at least three.  It follows that the edge $e'$ can be chosen so that $(T \cup \{f\}) - e'$ has strictly more leaves that $T$, a contradiction.  We conclude that there are at most $2 \ell$ components of $G - X$, and each component is a path for which every internal vertex of the path has degree in $G$ equal to two, as desired.
\end{proof}

\begin{LE}\label{lem:contract}
Let $G$ and $H$ be graphs.  Let $J$ be a subgraph of $G$ which is $2|E(H)|$-edge connected.  Let $G'$ be the graph obtained from $G$ by contracting $V(J)$ to a single vertex.  If $G'$ contains a strong immersion of $H$, then $G$ contains a strong immersion of $H$.  
\end{LE}

\begin{proof}
Let $v_J$ be the vertex of $V(G') \setminus V(G)$ corresponding to the contracted set $V(J)$.  Fix a strong immersion of $H$ in $G'$ given by the map $\pi$.  Let $\bar{G}$ be the graph obtained from $G'$ by deleting the edges incident to $v_J$ which are not used in the immersion, i.e., we delete the set of edges $X = \delta_{G'}(v_J) \setminus \bigcup_{f \in E(H)} E(\pi(f))$.  The graph $\bar{G}$ obviously still admits $H$ as a strong immersion.  Moreover, by the edge-connectivity of $J$, we observe that there exists a vertex of $V(J)$ with edge-disjoint paths in $G$ to $\delta_{\bar{G}}(v_J)$ in $G$ as $|\delta_{\bar{G}}(v_J)| \le 2|E(H)|$.  We conclude that $\bar{G}$ is contained as a strong immersion in $G$, completing the proof of the lemma.
\end{proof}

We now state the kind of decomposition we will encounter in the structure theorem for strong immersions. 

\begin{DEF}
Let $G$ be a graph on $n$ vertices.  A \emph{linear order} of $G$ is simply a labeling of the vertices $v_1, v_2, \dots, v_n$.  Fix a linear order $v_1, v_2, \dots, v_n$ of $G$.  For $2 \le i \le n-1$, let $T_i$ be the set of edges with one end in $\{v_1, \dots, v_{i-1}\}$ and the other end in $\{v_{i+1}, \dots, v_{n}\}$.  The \emph{hop-width} of the order $v_1, \dots, v_n$ is $\max_{2 \le i \le n-1} |T_i|$.  The \emph{hop-width} of $G$ is the minimum hop-width taken over all possible linear orders of $G$.
\end{DEF}
Our definition of hop-width is very similar to the definition of cut-width, but note that here edges incident to $v_i$ are not members of $T_i$.  Thus, the path $P_{k, t}$ has hop-width 0 but cut-width $k$.  As we have already observed, such paths with many parallel edges are a natural class of graphs which exclude strong immersions of bounded sized cliques.  This is the essential motivation for the definition of hop-width. In fact, one can see that if a highly edge connected graph has a linear order with bounded hop-width, then there should be many parallel edges between $v_i$ and $v_{i+1}$ for every $i$, that is, the graph is close to being a path with many parallel edges. 

We now state the structure theorem for highly edge connected graphs excluding a clique strong immersion. It shows that such graph can be built from a bounded number of parts having bounded hop-width; in a sense, this shows that paths with parallel edges are the canonical obstacles for clique strong immersions in highly edge connected graphs.

\begin{theorem}\label{thm:strong}  For all $t \ge 1$, and let $G$ be a $400t^5$-edge-connected graph with no strong $K_t$ immersion.  Then there exists a set $A \subseteq V(G)$ and $Z \subseteq E(G)$ such that 
\begin{enumerate}
\item $|A| \le 4t^2$, $|Z| \le 6t^{10}$, and
\item $(G-Z)-A$ has at most $2t^2$ distinct components. 
\end{enumerate}
Moreover, for every component $J$ of $(G-Z) - A$ with $n$ vertices, $J$ has a linear order $v_1, \dots, v_n$ satisfying the following:
\begin{enumerate}
\item[3.]  the only vertices of $J$ with a neighbor in $A$ (in $G-Z$) are $v_1$ and $v_n$, and 
\item[4.]  the linear order $v_1, \dots, v_n$ has hop-width at most $2t^6$ in $J$.
\end{enumerate}
\end{theorem}
\begin{proof}
The statement follows trivially for $t \le 2$.  Fix $t \ge 3$.  We define an auxiliary graph $R$ as follows.  Let $V(R) = V(G)$, and two vertices $x$ and $y$ in are adjacent in $R$ if and only if there exist at least $2t^2$ parallel edges of $G$ with ends $x$ and $y$.  We define a \emph{clump} to be a pair $(J, X)$ such that $J$ is an induced subgraph of $G$ and $X \subseteq V(J)$ with $X \neq \emptyset$ satisfying the following:
\begin{itemize}
\item[a.] If $|V(J)|\neq 1$, then $J$ is $2t^2$-edge connected;
\item[b.] Every component of $R$ is either contained in $V(J)$ or disjoint from $V(J)$;
\item[c.]  if $V(J) \setminus X \neq \emptyset$, then for every $x \in X$, there exist at least $2t^2$ distinct edges with one end equal to $x$ and other end in $V(J) \setminus X$; 
\item[d.] if $|V(J) \setminus X| \ge 2$, then $J-X$ is $2t^2$-edge connected and $|X| \ge 2$; 
\item[e.] if $V(J) \setminus X = \emptyset$, then $|V(J)| = |X| = 1$; 
\end{itemize}
Moreover, if we let $comp_R(J)$ be the number of components of $R$ contained in $V(J)$, then we have 
\begin{itemize}
\item[f.] if $|V(J)| \ge 3$, then $|X| \ge comp_R(J)+1$.
\end{itemize}

We first observe that for every component $R'$ of $R$, there exists a clump $(J, X)$ such that $V(J)= V(R')$.  The cases differ slightly, depending on $|V(R')|$, however, in each case we will fix $J = G[V(R')]$.  If $|V(R')| = 1$, then $J$ along with $X = V(R')$ satisfies the definition.  Similarly, if $|V(R')| = 2$, then let $X$ be an arbitrarily chosen vertex of $V(J)$ and $(J, X)$ satisfies the definition of a clump.  Finally, if $R'$ has at least three vertices, we consider a spanning tree of $R'$ and let $X$ be two leaves of the spanning tree.  Then $(J, X)$ again satisfies the definition.

Pick clumps $(J_i, X_i)$ for $1 \le i \le k$ for some positive $k$ such that $V(J_i)  \cap V(J_j)= \emptyset$ and $V(G) = \bigcup_1^k V(J_i)$.  Moreover, pick such clumps to minimize $k$.  By the previous paragraph, it is always possible to find such a collection of clumps.  

We first observe that there does not exist an index $i$ such that
$|X_i| \ge t$.  Otherwise, property c.~implies that after contracting
$V(J)-X$ to a single vertex, the contracted graph contains a strong
immersion of $S_{t, t}$ (defined in Section \ref{sec:bdeddeg}) and
then property d.~and Lemma \ref{lem:contract} imply that $J$ also contains
a strong immersion of $S_{t,t}$. Given that $S_{t,t}$ contains $K_t$
as a strong immersion, we conclude that $G$ contains a strong
immersion of $K_t$ as well, a contradiction.

\begin{Claim}\label{cl:2}
For distinct $i$ and $j$, there are at most $2t^4$ distinct edges with one end in $V(J_i)$ and one end in $V(J_j)$.
\end{Claim}
\begin{cproof}
Assume otherwise and that for $i, j$, $i \neq j$, we have $2t^4$ distinct edges with one end in $V(J_i)$ and one end in $V(J_j)$.  Given that $|X_i|$ and $|X_j|$ are both at most $t-1$, it follows that one of the following holds, up to swapping the indices $i$ and $j$: 
\begin{itemize}
\item[i.] there exists $x_i \in X_i$ and $x_j \in X_j$ with $2t^2$ parallel edges connecting them, or
\item[ii.] there exists $x_i \in X_i$ with at least $2t^2$ distinct edges to the set $V(J_j) \setminus X_j$, or
\item[iii.] there exist at least $2t^2$ distinct edges, each with one end in $V(J_i) \setminus X_i$ and one end in $V(J_j) \setminus X_j$.
\end{itemize}
It is easy to see that i.~cannot occur, as this would imply $x_i$ and $x_j$ are adjacent in $R$ and no component of $R$ has vertices in distinct clumps.  If iii.~occurs, then $\bar{J} = G[V(J_i) \cup V(J_j)]$ would form the clump $(\bar{J}, X_i \cup X_j)$, contrary to our choice to minimize the number of clumps covering $V(G)$.

Thus, we may assume that ii.~holds.  Note that $V(J_j) \setminus X_j$ cannot be a single vertex, lest there exist an edge of $R$ with ends in distinct clumps.  Thus, $|V(J_j) \setminus X_j| \ge 2$ and $|V(J_j)| \ge 4$.  Property f.~implies that $|X_j| \ge comp_R(J_j) + 1$.  We conclude that $(J_i \cup J_j, (X_i \cup X_j) - x_i)$ is a clump.  Note that f.~holds because $|(X_i \cup X_j) - x_i| = |X_i| + |X_j| -1 \ge comp_R(J_i) + comp_R(J_j) + 1 = comp_R(J_i \cup J_j) + 1$.  We conclude that if ii.~holds, we have a contradiction to our set of clumps to minimize $k$.   This completes the proof of the claim.
\end{cproof}

\begin{Claim}  $k = 1$ and we have exactly one clump.
\end{Claim}
\begin{cproof}
Assume $k \ge 2$.  Let $G_1$ be the graph obtained by contracting each of the vertex sets $J_i$ to a single vertex for all $1 \le i \le k$.  By property a.~and Lemma \ref{lem:contract}, we see that $G_1$ does not have a strong immersion of $K_t$.  Assume, to reach a contradiction, that $G_1$ has more than one vertex.  By Claim \ref{cl:2}, there are at most $2t^4$ parallel edges connecting any pair of vertices of $G_1$.  Let $G_2$ be the simple graph obtained from $G_1$ by deleting parallel edges.  In other words, $G_2$ is the simple graph with $V(G_2) = V(G_1)$ and two vertices of $G_2$ are adjacent if and only if they are adjacent in $G_1$.  By the edge connectivity of $G$, we see that $G_2$ has minimum degree $200t$.  But by Theorem \ref{thm:DDFMMS}, $G_2$ has a $K_t$ strong immersion, and consequently, $G_1$ does as well, a contradiction.
\end{cproof}

It follows that $R$ has at most $t$ components.  Note as well that $R$ does not contain a $K_{1,t}$ minor.  Otherwise, there would exist a $2t^2$-edge connected subgraph $J$ of $G$ and $t$ distinct vertices of $V(G) \setminus V(J)$ each with $2t^2$ edges to $V(J)$.  Thus, by Lemma \ref{lem:contract}, we conclude that $G$ would contain a strong immersion of $S_{t,t}$, a contradiction.  

By applying Lemma~\ref{lem:k1t} to each component of $R$, we see that
there exists a subset $A \subseteq V(R) = V(G)$ with $|A| \le 4t^2$
such that $R-A$ has at most $2t^2$ components, each of which is a
path.  Moreover, for every component $P$ of $R-A$, we have that
$N_R(A) \cap V(P)$ is a subset of the ends of $P$.

We fix the set $A$ for the remainder of the proof, and we let $k$ be a positive integer with the components of $R-A$ labeled $P_1, \dots, P_k$.  For each $P_i$, we let $n(i) = |V(P_i)|$ and let the vertices of $P_i$ be labeled $v_1^i, v_2^i, \dots, v_{n(i)}^i$.  Note $k \le 2t^2$.  

\begin{Claim}\label{cl:edgesbetweenpaths}
For all $1 \le i < j\le k$, there are at most $2t^6$ edges of $G$ with one end in $V(P_i)$ and one end in $V(P_j)$.
\end{Claim}
\begin{cproof}
  Assume, to reach a contradiction, that there are indices $i$ and $j$
  such that there exist at least $2t^6$ edges, each with one end in
  $P_i$ and one end in $P_j$.  Since no pair of vertices $x, y$ with
  $x \in V(P_i)$ and $y \in V(P_j)$ are adjacent in $R$, we know that
  there are at most $2t^2$ parallel edges from $x$ to $y$ in $G$.
  Thus, without loss of generality, we may assume that there are at
  least $t^2$ distinct vertices of $P_i$ which have a neighbor in
  $V(P_j)$.  By contracting $V(P_j)$ to a single vertex and applying
  Lemmas~\ref{lem:contract} and \ref{lem:pathplusvertex}, we conclude
  that $G$ contains a strong immersion of $K_t$, a contradiction.
\end{cproof}

We let $Z_1$ be the set of edges $e \in E(G)$ such that $e$ has ends in distinct paths $P_i$ and $P_j$.  By Claim~\ref{cl:edgesbetweenpaths} and the bound on $k$, we have that $|Z_1| \le 4t^{10}$.  

\begin{Claim}
Let $1 \le i \le k$.  Let $x, y$ be the endpoints of $P_i$.  Then there are at most $8t^6$ edges of $G$ with one end in $V(P) \setminus \{x, y\}$ and one in $A$.
\end{Claim}

\begin{cproof}
Assume, to reach a contradiction, that there exists an index $i$ such that if we let $x$ and $y$ be the endpoints of $P_i$, there are at least $8t^6$ edges of $G$ with one end in $A$ and the other end in $V(P_i) \setminus \{x, y\}$.  By the way $A$ was defined using Lemma \ref{lem:k1t}, there does not exist an edge of $R$ with one end in $A$ and one end in $V(P_i) \setminus \{x, y\}$.  Thus, for every pair of vertices $u \in A$ and $v \in V(P_i) \setminus\{x, y\}$, there are at most $2t^2$ parallel edges of $G$ from $u$ to $v$.  By assumption, there exist at least $4t^4 \ge t^2|A|$ distinct internal vertices of $P_i$ which have a neighbor in $A$.  Thus, there exists a vertex $v \in A$ and $t^2$ distinct internal vertices of $P_i$ which are each adjacent to $v$.  By Lemma~\ref{lem:pathplusvertex}, we conclude that $G$ contains a strong immersion of $K_t$, a contradiction.
\end{cproof}

Let $Z_2 \subseteq E(G)$ be given by $Z_2 = \{e \in E(G): \text{$\exists i$ such that $e \cap V(P_i) \setminus \{v_1^i, v_{n(i)}^i\} \neq \emptyset$ and $e \cap A \neq \emptyset$}\}$.  In other words, $Z_2$ is the set of edges with one end in $A$ and one end contained as an internal vertex of some $P_i$.  By the previous claim, $|Z_2| \le 8t^6k \le 16t^8$.  

We let $Z = Z_1 \cup Z_2$, and observe that $|Z| \le 6t^{10}$ by the assumption that $t \ge 3$.  By construction, the components of $(G-Z) - A$ are exactly the subgraphs of $G$ induced by $V(P_i)$ for $1 \le i \le k$.  Thus, we see that $A$ and $Z$ satisfy 1 and 2 in the statement of the theorem.  Moreover, by the fact that $Z$ contains every edge with one end in $A$ and one end in an internal vertex of $P_i$, we see that $A$ and $Z$ along with the linear order $v_1^i, \dots, v_{n(i)}^i$ of $V(P_i)$ satisfy 3.  Thus, to the complete the proof of the theorem, it suffices to show the following claim.

\begin{Claim}
For every $1 \le i \le k$, the linear order $v_1^i, \dots, v_{n(i)}^i$ of $G[V(P_i)]$ has hop-width at most $2t^6$.
\end{Claim}
\begin{cproof}
Assume the claim is false, and that there exists an index $a$ such that there are at least $2t^6$ edges of $G$ with one end in $\{v_1^i, \dots, v_{a-1}^i\}$ and one end in $\{v_{a+1}^i, \dots, v_{n(i)}^i\}$.  Given that $v_1^i, \dots, v_{n(i)}^i$ form an induced path in $R$, for all pairs of vertices $u, v$ with $u \in \{v_1^i, \dots, v_{a-1}^i\}$ and $v \in \{v_{a+1}^i, \dots, v_{n(i)}^i\}$, there are at most $2t^2$ parallel edges of $G$ with ends $u$ and $v$.  Thus, without loss of generality, there exist at least $t^2$ distinct vertices of $\{v_1^i, \dots, v_{a-1}^i\}$ that are adjacent to a vertex of $\{v_{a+1}^i, \dots, v_{n(i)}^i\}$. 
By contracting $\{v_{a+1}^i, \dots, v_{n(i)}^i\}$ (Lemma~\ref{lem:contract}) and invoking Lemma~\ref{lem:pathplusvertex}, we conclude that $G$ contains a strong immersion of $K_t$, a contradiction.  
\end{cproof}

This final claim shows that conclusion 4.~in the statement of the theorem holds, completing the proof.
\end{proof}

We conclude with the observation that if a graph $G$ has subsets $A$ and $Z$ satisfying 1-4 in the statement of Theorem \ref{thm:strong} for some value $t$, then the graph $G$ does not have a strong immersion of $K_{10t^{10}}$, indicating that Theorem \ref{thm:strong} is a good characterization of graphs excluding a strong immersion of a clique.

\begin{center}
{\bf ACKNOWLEDGEMENTS}
\end{center}

The authors gratefully acknowledges Paul Seymour's valuable contributions at the early stages of this work, and specifically his helpful observations on how to use the line graph and tangles in the proof of Theorem \ref{thm:main}.  Thanks as well to Chun-Hung Liu for pointing out an error in the proof of Theorem \ref{thm:imp2} in an earlier version of this article.

\end{document}